\documentclass[10pt]{amsart}
\usepackage{amsthm, amsmath,cleveref, amsfonts, amssymb, enumerate, textcmds, tensor, enumitem, mathdots}
\usepackage{pst-node}
\usepackage{tikz-cd} 
\usepackage{graphicx,tikz}
\usepackage{enumerate}
\usepackage{pinlabel}
\usepackage[colorinlistoftodos]{todonotes}
\newcommand{\R}{\mathbb R}

\newtheorem{thm}{Theorem}[section]
\newtheorem{lem}[thm]{Lemma}
\newtheorem{prop}[thm]{Proposition}

\newtheorem*{thma}{Theorem A}
\newtheorem*{thma'}{Theorem A'}
\newtheorem*{exa}{Example A}
\newtheorem*{prob}{Problem A}
\newtheorem*{thmb}{Theorem B}

\theoremstyle{definition}
\newtheorem{defn}[thm]{Definition}
\newtheorem{remark}[thm]{Remark}

\begin{document}
\title{Minimal diffeomorphisms with $L^1$ Hopf differentials}
\author{Nathaniel Sagman}
\begin{abstract}
     We prove that for any two Riemannian metrics $\sigma_1, \sigma_2$ on the unit disk, a homeomorphism $\partial\mathbb{D}\to\partial\mathbb{D}$ extends to at most one quasiconformal minimal diffeomorphism $(\mathbb{D},\sigma_1)\to (\mathbb{D},\sigma_2)$ with $L^1$ Hopf differential. For minimal Lagrangian diffeomorphisms between hyperbolic disks, the result is known, but this is the first proof that does not use anti-de Sitter geometry. We show that the result fails without the $L^1$ assumption in variable curvature. The key input for our proof is the uniqueness of solutions for a certain Plateau problem in a product of trees.
\end{abstract}

\maketitle

\begin{section}{Introduction}
\subsection{Minimal diffeomorphisms}
Throughout, $\mathbb{D}\subset \mathbb{C}$ denotes the unit disk. A diffeomorphism between Riemannian surfaces $f:(\Sigma,\sigma_1)\to (\Sigma,\sigma_2)$ is a minimal diffeomorphism if its graph inside the product $4$-manifold $(\Sigma^2,\sigma_1\oplus\sigma_2)$ is a minimal surface. Schoen first proved that any diffeomorphism between closed hyperbolic surfaces can be deformed to a unique minimal one \cite{Sc} (see also \cite{Wa} and \cite[Theorems B1 and C]{MSS}). Bonsante-Schlenker then proved the following striking result \cite{BS}.
\begin{thm}[Bonsante-Schlenker]\label{BS}
    Let $\varphi:\partial\mathbb{D}\to\partial \mathbb{D}$ be a quasisymmetric map and let $\sigma$ be a hyperbolic metric on $\mathbb{D}$. There exists a unique minimal diffeomorphism $f:(\mathbb{D},\sigma)\to (\mathbb{D},\sigma)$ that extends to $\varphi$ on $\partial\mathbb{D}.$ Moreover, $f$ is quasiconformal.
\end{thm} 
Note that any graph of a minimal diffeomorphism in $(\mathbb{D}^2,\sigma\oplus \sigma)$ is necessarily Lagrangian. Three more proofs of Theorem \ref{BS} have appeared in the literature \cite{BSe}, \cite{LTW}, \cite{SST}; including \cite{BS}, all of the proofs have exploited a correspondence between minimal Lagrangian graphs in $(\mathbb{D}^2,\sigma\oplus \sigma)$ and spacelike maximal surfaces in $3d$ anti-de Sitter space, $\textrm{AdS}^3$. The $\textrm{AdS}^3$ perspective reveals a lot of interesting mathematics and deserves further exploration, but perhaps some of the geometry of minimal diffeomorphisms becomes obscured.

In this paper we give a short and purely Riemannian proof of the uniqueness statement, albeit under an additional integrability assumption, but which also applies in a broader setting. The proof is a manifestation of the basic fact that among all maps from the unit disk that fill up a Jordan domain $\Omega\subset \mathbb{C},$ the Riemann map minimizes the Dirichlet energy.

Before stating our main theorem, we should recall that Markovi{\'c} proved in \cite{M1} that any minimal diffeomorphism between closed Riemannian surfaces of genus at least $2$ is unique in its homotopy class, interestingly with no requirement on the curvature of the metrics. Previously, stability and uniqueness of minimal diffeomorphisms had been studied by Wan in \cite{Wa} under additional curvature assumptions (see also \cite[Proposition 1.7]{DL}). The results of the current paper are inspired by the paper \cite{M1}, along with the older paper \cite{MM} of Mateljevi{\'c}-Markovi{\'c} about harmonic maps of the disk. 

Let $\sigma_1,\sigma_2$ be a pair of Riemannian metrics on the unit disk. Given the graph of a minimal map $G\subset (\mathbb{D}^2,\sigma_1\oplus \sigma_2)$, the projections $p_i:G\to (\mathbb{D},\sigma_i),$ $i=1,2,$ are harmonic. Any harmonic map from a surface comes with a holomorphic quadratic differential called the Hopf differential, and the Hopf differentials of $p_1$ and $p_2$ necessarily sum to zero. Conversely, a pair of harmonic diffeomorphisms $h_i:\mathbb{D}\to (\mathbb{D},\sigma_i)$ with opposite Hopf differentials determines a minimal diffeomorphism via $h_1\circ h_2^{-1}.$  We refer to the Hopf differential of the first harmonic projection as the Hopf differential of the minimal diffeomorphism.

\begin{thma}
Let $\sigma_1$ and $\sigma_2$ be a pair of Riemannian metrics on $\mathbb{D}$ and let $\varphi:\partial\mathbb{D}\to\partial\mathbb{D}$ be a homeomorphism. Then $\varphi$ extends to at most one quasiconformal minimal diffeomorphism $(\mathbb{D},\sigma_1)\to(\mathbb{D},\sigma_2)$ with integrable Hopf differential.
\end{thma}
 Theorem A has no curvature assumption, as in \cite{M1}, and no completeness either. We do not treat the question of existence, although we expect existence to hold for complete metrics with pinched negative curvature.  Existence results can be found in \cite{Bre} and \cite{Ta}.
\begin{remark}\label{Bers}
When the metric is complete and has pinched negative curvature, a harmonic map is quasiconformal, and hence admits quasisymmetric boundary extension, if and only if the Hopf differential has finite Bers norm \cite{Wan1}, \cite{TW}, which is weaker than integrability (see \cite[Chapter 5.4]{Hu}). So for $\sigma_1,\sigma_2$ hyperbolic, Theorem A is contained in Theorem \ref{BS}.
\end{remark}
\begin{remark}
    The quasisymmetric assumption in Theorem \ref{BS} can be removed. The proof of existence in \cite{BS} doesn't appear to use quasisymmetry, although uniqueness certainly does. More recent work of Seppi-Smith-Toulisse \cite{SST} implies that any homeomorphism of $\partial \mathbb{D}$ extends uniquely to a minimal diffeomorphism of the hyperbolic disk whose associated maximal surface in $\textrm{AdS}^3$ is complete. Trebeschi shows in \cite{Tre} that the required completeness assumption from \cite{SST} is automatic, thereby closing the uniqueness question.
\end{remark}
We also observe that at the chosen level of generality, a boundedness assumption like $L^1$-ness is necessary. 
\begin{exa}
    There exists a flat metric $g$ on $\mathbb{D}$ and two quasiconformal minimal diffeomorphisms of $(\mathbb{D},g)$ that extend to the same homeomorphisms of $\partial\mathbb{D}.$ One of them has $L^1$ Hopf differential, while the other does not.
\end{exa}
So, interestingly, our assumption is a trade-off: by Remark \ref{Bers} it is restrictive in hyperbolic space, but necessary in our general setting. Without such an assumption, our method still proves that minimal diffeomorphisms are stable (see Definition \ref{stabilitydef}). 
\begin{thma'}
    Let $\sigma_1,\sigma_2$ be a pair of Riemannian metrics on the unit disk. Any minimal diffeomorphism $(\mathbb{D},\sigma_1)\to(\mathbb{D},\sigma_2)$ is stable.
\end{thma'}

The geometric meaning of the $L^1$ assumption will be explained in the coming subsections. For now, it  allows us to invoke the new main inequality, introduced in \cite{M1} and \cite{M2} and studied further in \cite{MS}.
\subsection{The new main inequality}
Let $S$ be a Riemann surface, $\phi_1,\dots,\phi_n$ integrable holomorphic quadratic differentials on $S$ summing to zero, and $f_1,\dots, f_n:S\to S'$ mutually homotopic quasiconformal maps to another Riemann surface with Beltrami forms $\mu_1,\dots,\mu_n$. If $\partial S$ is non-empty, we ask that $f_1,\dots, f_n$ are mutually homotopic relative to $\partial S$.
\begin{defn}
    We say that the new main inequality holds if: 
 \begin{equation}\label{newmain1}
   \textrm{Re}\sum_{i=1}^n \int_S \phi_i\cdot \frac{ \mu_i}{1-|\mu_i|^2} \leq \sum_{i=1}^n \int_S |\phi_i|\cdot \frac{|\mu_i|^2}{1-|\mu_i|^2}.
 \end{equation}
\end{defn}
Taking $n=2$, $f_1$ (relatively) homotopic to the identity, and $f_2$ the identity, (\ref{newmain1}) becomes the classical main inequality for quasiconformal maps, a key ingredient in the proof of Teichm{\"u}ller's uniqueness theorem. On closed surfaces, (\ref{newmain1}) always holds for $n=1,2$ \cite{M1}, but fails for $n\geq 3$ \cite{M2}. We prove
\begin{thmb}
    For $n=2$ and $S=\mathbb{D},$ the new main inequality always holds. 
\end{thmb}

\subsection{Outline of the argument}
The proof of Theorem A follows the methodology of a series of papers on the uniqueness and non-uniqueness of minimal surfaces in products of closed surfaces \cite{M1}, \cite{M2}, \cite{MSS}. These papers led to the paper \cite{SS}, which resolved a conjecture about minimal surfaces in more general locally symmetric spaces.

In \cite{M1}, Markovi{\'c} proves uniqueness of minimal diffeomorphisms using the new main inequality for $n=2.$ The main difference between our argument and that of \cite{M1} is the proof of the new main inequality. In \cite{M1}, (\ref{newmain1}) is proved as a consequence of the known uniqueness of minimal Lagrangians between closed hyperbolic surfaces. Our proof, on the other hand, does not assume any other uniqueness result. One of our motivations for writing this paper is to share what we think is the most natural proof.

\subsubsection{Theorem A} There is an established theory of harmonic maps from Riemannian manifolds to complete and non-positively curved (NPC) length spaces, such as $\R$-trees. A holomorphic quadratic differential on a Riemann surface gives rise to an $\R$-tree and a harmonic map from (a cover of) the Riemann surface to that $\R$-tree called a leaf-space projection. 

Given a minimal diffeomorphism, we get a minimal map $h$ into a product of two disks. As well, the Hopf differentials determine a minimal map $\pi$ into a product of two $\R$-trees. The $L^1$ assumption is equivalent to demanding that the image of $\pi$ has finite total energy (see the definition in Section 2.1). If $\mathbb{D}_r=\{z\in \mathbb{C}: |z|<r\}$, we show that the difference between the energy of $h|_{\mathbb{D}_r}$ and that of any other map with the same asymptotic boundary data is strictly bounded below by the difference in energies between $\pi$ and some other map to the same product of $\R$-trees (we're lying a little bit to keep things simpler here). Theorem B implies that the total energy of $\pi$ is less than that other map. We deduce that for $r$ large, any such $h$ strictly minimizes the energy over $\mathbb{D}_r$, and hence $h$ is unique.

\subsubsection{Theorem B} To prove Theorem B, we set up a Plateau problem in a product of $\R$-trees, asking for maps from $\overline{\mathbb{D}}$ that take $\partial\mathbb{D}$ onto a fixed embedded circle and minimize a certain energy or area. We will observe that if a product of leaf-space projections solves the Plateau problem, then (\ref{newmain1}) holds for their associated quadratic differentials. Thus, Theorem B amounts to studying this Plateau problem for $n=2.$

Let $\mathbb{A}^1(\mathbb{D})$ be the Banach space of integrable holomorphic quadratic differentials on the unit disk equipped with the $L^1$ norm. Since (\ref{newmain1}) is a closed inequality, it is routine to prove
 \begin{prop}\label{densityarg}
     Let $B\subset \mathbb{A}^1(\mathbb{D})$ be a dense subset, and suppose that for all quasiconformal maps that agree on $\partial\mathbb{D},$ (\ref{newmain1}) holds for every $\phi\in B.$ Then (\ref{newmain1}) holds for all $\phi\in \mathbb{A}^1(\mathbb{D}).$
 \end{prop}
 Thus, while $\R$-trees can be quite wild (see the images in \cite{ATW}), we only need to study the Plateau problem for $\R$-trees arising from a generic class in $\mathbb{A}^1(\mathbb{D})$. Choosing our quadratic differentials to be polynomials, the product of $\R$-trees is a $2$-dimensional simplicial complex, and retracts onto the image of the leaf-space projections. Then the Plateau problem becomes nearly identical to the (more or less trivial) Plateau problem in the plane.

\subsection{Future perspectives}
In view of Theorem \ref{BS} and Remark \ref{Bers}, it's natural to pose the following problem, which we'll say more about in Section \ref{counter}.
\begin{prob}
    For $\sigma_1$ and $\sigma_2$ complete and with pinched negative curvature, prove uniqueness of minimal diffeomorphisms $(\mathbb{D},\sigma_1)\to (\mathbb{D},\sigma_2)$ with quasisymmetric boundary data.
\end{prob}
In another direction, the ideas here might be applicable for studying minimal disks in rank $2$ symmetric spaces. For all of the rank $2$ Lie groups that admit higher Teichm{\"u}ller spaces, except the split real form of $G_2,$ various authors have established existence and uniqueness results for quasi-isometric minimal disks in associated pseudo-Riemannian homogeneous spaces (see \cite{BH}, \cite{LTW}, and \cite{LT}). Through some form of a twistor correspondence, akin to using $\textrm{AdS}^3$ to prove Theorem \ref{BS}, these results imply existence and uniqueness results for minimal surfaces in symmetric spaces. One could try to push further the ideas from this paper to give a unified set of uniqueness statements, without passing through pseudo-Riemannian geometry. In the symmetric space context, products of trees are replaced by rank $2$ Euclidean buildings modeled on Cartan subalgebras.

The $L^1$ condition in fact generalizes nicely. As explained in \cite{SS} (for closed surfaces), a minimal map to a symmetric space of rank $k$ gives rise to a minimal map to $\mathbb{R}^k$ with special properties, in a rather explicit way. For $\mathbb{D}^2$, the construction is less complicated and we can explain it here. Given a minimal diffeomorphism from $\mathbb{D}$ to $\mathbb{D}$ (with any metrics), assume that the Hopf differential $\phi$ is the square of an abelian differential $\alpha$. Then the differentials $\alpha$ and $i\alpha$ determine Weierstrass-Enneper data for a minimal map to $\R^2$. The energy of this map, which dominates the area of the image, is equal to the $L^1$ norm of $\phi$. If $\phi$ is not a square, there is a degree $2$ branched cover of $\mathbb{D}$ on which it lifts to one, and after passing to the universal cover of that branched cover, we can repeat the construction. Going back to symmetric spaces, one could look at minimal maps that give finite area maps to $\R^2$.

Finally, this paper also opens up some new discussions on harmonic maps from the disk to $\R$-trees that we would like to highlight (see Remarks \ref{non-convex} and \ref{filling}).

\subsection{Acknowledgements}
 I'd like to thank Peter Smillie and Enrico Trebeschi for helpful discussions related to this paper. I'm grateful to Pallavi Panda for making the figures. I'd also like to thank the anonymous referee for many helpful comments that improved the paper, especially the proof of Theorem B. This work was supported by Luxembourg National Research Fund (Fonds National de la Recherche, FNR) grant number O20/14766753, \it{Convex Surfaces in Hyperbolic Geometry.}

\end{section}

\section{Harmonic maps}
We give all definitions specialized to open subsets of $\mathbb{C}$. Throughout, let $z=x+iy$ be the standard coordinate on $\mathbb{C}.$ For this section, fix a simply connected domain $\Omega\subset \mathbb{C}$.
\subsection{Harmonic maps}
Let $(M,d)$ be a complete and NPC length space and $h:\Omega\to (M,d)$ a locally Lipschitz map. By work of Korevaar-Schoen \cite[Theorem 2.3.2]{KS}, we can associate a locally $L^1$ measurable metric $g(h)$, defined on pairs of Lipschitz vector fields. If $h$ is a $C^1$ map to a smooth manifold $M$ and the distance $d$ is induced by a Riemannian metric $\sigma$, then $g(h)$ is represented by the pullback metric $h^*\sigma$. The energy density form is the locally $L^1$ form 
\begin{equation}\label{en}
    e(h)=\frac{1}{2}(\textrm{trace} g(h))dx dy,
\end{equation}
where the trace is with respect to the flat metric on $\Omega.$ The total energy is $$\mathcal{E}(\Omega,h) = \int_\Omega e(h).$$
We want to consider boundary value problems and to allow $\mathcal{E}(\Omega,h)=\infty.$ Hence we make
\begin{defn}\label{harmdef}
$h$ is harmonic if for all relatively compact domains $U\subset \Omega$ with compact closure contained in $\Omega$, $h|_U$ is a critical point of $f\mapsto \mathcal{E}(U,f)$ with respect to variations compactly supported on $U.$
\end{defn}
There are many existence and uniqueness results that extend the Riemannian theory. We recall just one.
\begin{thm}[Theorem 2.2 in \cite{KS}]\label{KSthm}
    Let $U\subset \Omega$ be a relatively compact and Lipschitz domain with compact closure contained in $\Omega$. Given a Lipschitz map $f:\overline{U}\to (X,d)$, there exists a unique harmonic map $h:\overline{U}\to X$ such that $h=f$ on $\partial U$. Moreover, $h$ minimizes $\mathcal{E}(U,\cdot)$ relative to its boundary values.
\end{thm}
Writing $x_1=x$ and $x_2=y$, let $g_{ij}(h)$ be the components of $g(h)$. 

\begin{defn}
    The Hopf differential of $h$ is the measurable quadratic differential given by
\begin{equation}\label{mhopf}
\phi(h)(z)=\frac{1}{4}(g_{11}(h)(z)-g_{22}(h)(z)-2ig_{12}(h)(z))dz^2.
\end{equation}
\end{defn}
When $h$ is harmonic, the Hopf differential is represented by a holomorphic quadratic differential. In the Riemannian setting, (\ref{mhopf}) is 
$$
\phi(h)(z) = h^*\sigma\Big (\frac{\partial}{\partial z},\frac{\partial}{\partial z}\Big )(z)dz^2,$$ and hence $h$ is weakly conformal if and only if $\phi(h)=0.$ Recall that the image of a harmonic and conformal immersion inside a Riemannian manifold defines a minimal surface. Thus, the following definition is appropriate for the singular setting.
\begin{defn}
   $h:\Omega\to (M,d)$ is minimal if it is harmonic and $\phi(h)=0.$ 
\end{defn}
Directly from the definitions (\ref{en}) and (\ref{mhopf}), the energy density and Hopf differential of a map into a product of spaces is the sum of the energies and Hopf differentials of the component maps. Consequently, a map into a product of disks $$h=(h_1,h_2):\Omega\to (\mathbb{D}^2,\sigma_1\oplus\sigma_2)$$ is harmonic if both components are harmonic, and minimal when $\phi(h_1)=-\phi(h_2).$ It should thus be clear that minimal diffeomorphisms yield minimal graphs in the product $4$-manifold.

\subsection{Harmonic maps to $\R$-trees}\label{dualtrees}
For an introduction to harmonic maps to $\R$-trees in the more usual context of closed surfaces, see \cite{Wf}.
    \begin{defn}
An $\mathbb{R}$-tree is a length space $(T,d)$ such that any two points are connected by a unique arc, and every arc is a geodesic, isometric to a segment in $\mathbb{R}$.
\end{defn}
Let $\phi$ be a holomorphic quadratic differential on $\mathbb{C}$. The vertical (resp. horizontal) foliation of $\phi$ is the singular foliation on $\mathbb{C}$ whose non-singular leaves are the integral curves of the line field on $\Omega\backslash \phi^{-1}(0)$ on which $\phi$ returns a negative (resp. positive) real number. The singularities are standard prongs at the zeros. Both foliations come with transverse measures $|\textrm{Re}\sqrt{\phi}|$ and $|\textrm{Im}\sqrt{\phi}|$  respectively (see \cite[Expos{\'e} 5]{Thbook} for precise definitions).  

 We define an equivalence relation on $\mathbb{C}$ under which two points $x,y\in\mathbb{C}$ are identified if they lie on the same leaf of the vertical foliation of $\phi$. We denote the quotient by $T$, and the quotient projection as $\pi: \mathbb{C}\to T$. The transverse measure pushes down via $\pi$ to a distance function $d$ that makes $(T,d)$ into an $\R$-tree. The same construction applies for the horizontal foliation, which is also the vertical foliation of $-\phi$. We work with the leaf space for the vertical foliation, unless specified otherwise.

The energy density and Hopf differential of $\pi$ can be described explicitly: at a point $p\in\mathbb{C}$ on which $\phi(p)\neq 0$, $\pi$ locally isometrically factors through a segment in $\mathbb{R}$. In a small neighbourhood around that point, $g(h)$ is represented by the pullback metric of the locally defined map to $\mathbb{R}$. Therefore, it can be computed that
\begin{equation}\label{enho}
    e(\pi)=|\phi|/2, \hspace{1mm} \phi(\pi)=\phi/4.
\end{equation}
In view of (\ref{enho}), we will always rescale the metric on $T$ from $(T,d)$ to $(T,2d).$ 

Now, recall that we've been working on a simply connected domain $\Omega\subset \mathbb{C}.$ We add the assumption that $\Omega$ is bounded and has Lipschitz boundary, and restrict $\pi$ from $\mathbb{C}$ to $\overline{\Omega}\subset \mathbb{C}$. In our normalization, by the first formula in (\ref{enho}) the total energy of $\pi$ on $\Omega$ is 
$$\mathcal{E}(\Omega,\pi)=\int_\Omega |\phi|<\infty,$$
the $L^1(\Omega)$-norm of $\phi.$ We will think of $\pi: \overline{\Omega}\to (T,2d)$ as solving a boundary value problem, which is justified by the proposition below.
\begin{prop}
    $\pi: \overline{\Omega}\to (T,2d)$ is harmonic.
\end{prop}
\begin{proof}
From (\ref{enho}), $\pi|_{\overline{\Omega}}$ is Lipschitz. Thus, we can appeal to the existence result of Korevaar-Schoen, Theorem \ref{KSthm}, which produces a harmonic map $h:\overline{\Omega}\to (T,2d)$ that agrees with $\pi$ on $\partial\Omega$. It is proved in \cite[Proposition 3.2]{FW} that harmonic maps to $\R$-trees pull back germs of convex functions to germs of subharmonic functions. On the other hand, the local argument from \cite[Section 4]{Wf} shows that $\pi$ pulls back germs of convex functions on $(T,2d)$ to germs of subharmonic functions. As the distance function $d$ is convex, the map $p\mapsto d(h(p),\pi(p))$ is subharmonic on $\Omega.$ Since $d(h(p),\pi(p))=0$ on $\partial \Omega,$ we deduce that $h=\pi.$
\end{proof}

Later on, we will add a condition that will simplify the proofs: that $\phi$ is a polynomial. Through basic geometric considerations, the following should be clear. Details are explained in \cite[Proposition 2.2]{AW}.
 \begin{prop}\label{auwan}
   When $\phi$ is a polynomial of degree $n,$ the leaf spaces of the vertical and horizontal foliations are complete simplicial $\R$-trees with $n+2$ infinite edges, which can be properly and geodesically embedded in $\R^2.$
\end{prop}
See Figure \ref{fig:fig1}. In fact, the converse to Proposition \ref{auwan} is the subject of the paper \cite{AW}.

 \begin{figure}[ht]
     \centering
     \includegraphics[scale=0.25]{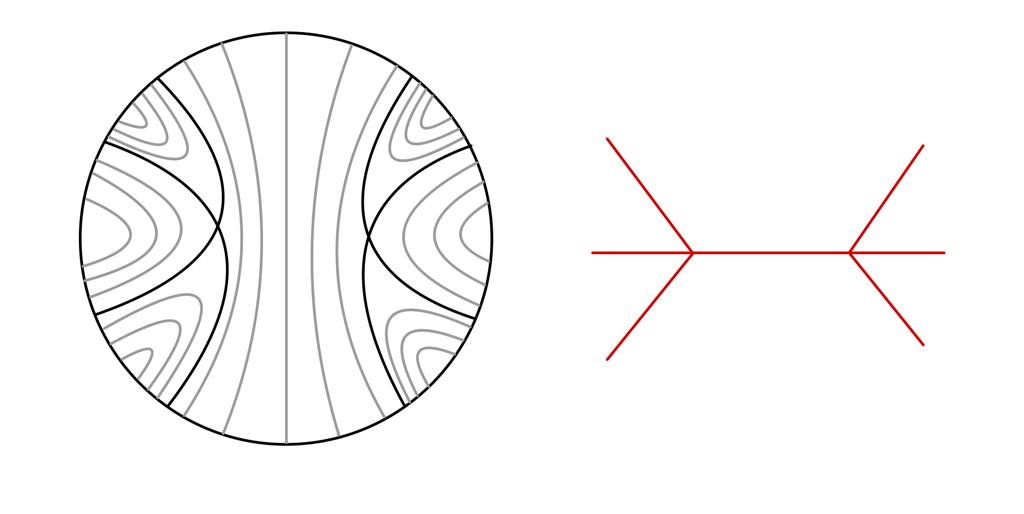}
    \caption{The foliation and $\R$-tree of a polynomial with two double order zeros, restricted to $\overline{\mathbb{D}}$.}
     \label{fig:fig1}
 \end{figure} 

\begin{remark}\label{non-convex}
    Since our theorem concerns $\phi\in \mathbb{A}^1(\mathbb{D}),$ it might appear strange to begin with a quadratic differential on $\mathbb{C}.$
    Our main reason for making this choice is that the $\R$-tree construction requires more care on the disk, and it's unnecessarily complicated to define a Plateau problem as below for quadratic differentials that don't extend holomorphically past $\partial\mathbb{D}.$ It would be interesting to develop the theory of harmonic maps from $\mathbb{D}$ to $\R$-trees more thoroughly.
\end{remark}

\section{Proofs}
\subsection{The Reich-Strebel formula}
The Reich-Strebel formula (Proposition \ref{RSorig} and originally equation 1.1 in \cite{RS}) will play an important role in the main proofs. 

As above, let $\Omega,\Omega'\subset \mathbb{C}$ be open.
Recall that the Beltrami form of a locally quasiconformal map $f:\Omega\to\Omega'\subset\mathbb{C}$ is the measurable tensor $$\mu(z)=\frac{f_{\overline{z}}(z)d\overline{z}}{f_z(z)dz},$$ where derivatives are taken in the weak sense. 
Let $(M,d)$ be a complete NPC space, and $h:\Omega\to (M,d)$ a locally Lipschitz map with finite total energy. Let $\mu_f$ be the Beltrami form of $f$, $J_{f^{-1}}$ the Jacobian of $f^{-1}$, and $\phi$ the Hopf differential of $h$, which need not be holomorphic. One can verify the identity 
\begin{align*}
     e(h\circ f^{-1})&=(e(h)\circ f^{-1})J_{f^{-1}}+2(e(h)\circ f^{-1})J_{f^{-1}} \frac{(|\mu_f|^2\circ f^{-1})}{1-(|\mu_f|^2\circ f^{-1})} \\
     &-4\textrm{Re}\Big ( (\phi(h)\circ f^{-1})J_{f^{-1}}\frac{(\mu_f\circ f^{-1})}{1-(|\mu_f|^2\circ f^{-1})}\Big )
\end{align*}
from \cite{RS}. Integrating, we obtain the proposition below.
\begin{prop} We have the formula
\begin{equation}\label{RSorig}
    \mathcal{E}(\Omega',h\circ f^{-1}) -\mathcal{E}(\Omega,h) =  \int_\Omega -4\textrm{Re} \Big (\phi(h)\cdot \frac{ \mu_f}{1-|\mu_f|^2} \Big )+ 2 e(h)\cdot \frac{|\mu_f|^2}{1-|\mu_f|^2}.
\end{equation}
\end{prop}
When the target $(M,d)$ is an $\R$-tree, we replace $e(h)$ by $2|\phi(h)|,$ so the formula (\ref{RSorig}) involves only $\phi$ and $\mu$. 

\subsection{The Plateau problem in a product of trees}\label{plateausection}
See \cite{Nit} or \cite[Chapter 4]{CM} for exposition on the classical Plateau problem about minimal surfaces in $\R^n$ (among many other excellent sources).

For $n\geq 2$ and $i=1,\dots, n$, let $(T_i,d_i)$ be a geodesically complete $\R$-tree and let $(X,d)$ be the product $(X,d)=\prod_{i=1}^n (T_i,d_i).$ Let $\gamma\subset (X,d)$ be an embedded circle.
\begin{defn}
    Let $h:\overline{\mathbb{D}}\to (X,d)$ be a locally Lipschitz map such that $h(\partial\mathbb{D})=\gamma$. We say that $h$ solves the Plateau problem for $\gamma$ if for all Lipschitz maps $g:\overline{\mathbb{D}}\to (X,d)$ such that $g|_{\partial\mathbb{D}}$ differs from $h$ by a reparametrization of $\partial\mathbb{D},$ $$\mathcal{E}(\mathbb{D},h)\leq \mathcal{E}(\mathbb{D},g).$$
\end{defn}
\begin{remark}
    If $h$ solves the Plateau problem, then it should surely minimize the area among a larger class of maps. We keep such a strong condition on $g$ in order to give a smoother proof.
\end{remark}

The proposition below explains the relation between our Plateau problem and the new main inequality.
\begin{prop}\label{newmaincrit}
Suppose that $h=(h_1,\dots,h_n):\overline{\mathbb{D}}\to (X,d)$ solves the Plateau problem for $\gamma$. For $i=1,\dots, n$, let $\phi_i$ be the Hopf differential of $h_i$. Then, for any choice of quasiconformal maps $f_i:\mathbb{D}\to\mathbb{D}$ that all extend to the same map on $\partial \mathbb{D}$, (\ref{newmain1}) holds. 
\end{prop}
\begin{proof}
For each $i,$ let $\mu_i$ be the Beltrami form of $f_i.$ Let $h'=(h_1\circ f_1^{-1}, \dots, h_n\circ f_n^{-1})$. Since every $f_i$ agrees on $\partial{\mathbb{D}}$, $h|_{\partial\mathbb{D}}$ and $h'|_{\partial\mathbb{D}}$ differ by a reparametrization of $\partial\mathbb{D}$. The hypothesis that $h$ solves the Plateau problem thus yields that $$\mathcal{E}(\mathbb{D},h) \leq \sum_{i=1}^n \mathcal{E}(\mathbb{D},h_i\circ f_i^{-1}).$$ Splitting $\mathcal{E}(\mathbb{D},h)$ into a sum of energies,  
\begin{equation}\label{something}
    \sum_{i=1}^n \mathcal{E}(\mathbb{D},h_i) -\sum_{i=1}^n \mathcal{E}(\mathbb{D},h_i\circ f_i^{-1})\leq 0.
\end{equation}
For each $i$, we apply the Reich-Strebel formula to (\ref{something}) and insert (\ref{enho}) for $e(h_i).$ Dividing by $4$ returns $$\sum_{i=1}^n\textrm{Re} \int_S \phi_i\cdot \frac{ \mu_i}{1-|\mu_i|^2}\leq \sum_{i=1}^n\int_S |\phi_i|\cdot \frac{|\mu_i|^2}{1-|\mu_i|^2}.$$ 
\end{proof}
\begin{remark}
    A version of Proposition \ref{newmaincrit} holds for maps between closed surfaces. In that context, we proved the converse statement \cite[Theorem A]{MS}.
\end{remark}
Define the Korevaar-Schoen area as $$A(h) = \int_{\mathbb{D}}\sqrt{\det g(h)}.$$ When $(X,d)$ can be isometrically embedded in a manifold, $h$ is injective, and $g(h)$ is smooth and non-degenerate on a residual set, then $A(h)$ is the area of the image of $h$. Note that, unlike the energy, $A(h)$ is invariant under precomposition of $h$ by diffeomorphisms. By an application of Cauchy-Schwarz,
$$A(h)\leq \mathcal{E}(\mathbb{D},h),$$
with equality if and only if $\phi(h)=0$. 
\begin{remark}\label{areminimal}
    Plateau problem solutions are minimal maps and minimize $A(\cdot)$ (and interestingly, the analogous property is not true for more general proper metric spaces \cite{LWenger}). Indeed, by local isometric factoring, the Korevaar-Schoen metric of a harmonic map to $(X,d)$ is smooth off of a discrete set. From that observation, the proof for maps to $\R^n$ goes through (see \cite[Chapter 4.1]{CM}).    
\end{remark}

\begin{remark}\label{filling}
   Minimal maps to $(X,d)$ with fixed boundary curve do not always exist (take $n$ copies of the same tree). In view of Theorem \ref{KSthm}, one could hope to construct a minimal map by adapting the Douglas-Rado strategy for the classical Plateau problem. We expect this method to have a better chance of working when the trees are dual to foliations that have a ``filling" property (see \cite{We}).
\end{remark}

\subsection{The Plateau problem from a polynomial}
 We now set $n=2$. In view of Proposition \ref{densityarg}, we concern ourselves only with the Plateau problem for products of simplicial trees. Let $\phi$ be a polynomial on $\mathbb{C}$ with simplicial trees $(T_1,2d_1)$ and $(T_2,2d_2)$ arising from the vertical and horizontal foliations, and let $$\pi=(\pi_1,\pi_2):\mathbb{C}\to (X,d)=(T_1,2d_1)\times (T_2,2d_2)$$ be the product of the leaf-space projections. We will show that $\pi|_{\partial\mathbb{D}}$ is an embedded circle, so that $\pi|_{\overline{\mathbb{D}}}$ defines a Plateau problem, and that $\pi|_{\overline{\mathbb{D}}}$ solves that problem.
 \begin{lem}
     $\pi$ defines a homeomorphism from $\mathbb{C}$ onto its image.
 \end{lem}
 \begin{proof}
Note that for points $x,y\in \mathbb{C}$, the condition that $\pi(x)=\pi(y)$ is equivalent to the statement that $x$ and $y$ lie on the same leaf of both the vertical folation and the horizontal foliation.
 
Observe that $\pi$ is a local homeomorphism onto its image: at any point $p\in \mathbb{C},$ we can choose a local coordinate $w$ with $w(p)=0$ and in which $\phi$ is given by $w^n dw^2$. In such a coordinate, for $n>0$ both foliations have critical trajectories that are $(n+2)$-pronged stars, and for $n=0$ we have ordinary foliation charts, and from these local pictures it is easy to verify that $\pi$ is an embedding in a neighbourhood of $p.$ As well, $\pi$ is a proper map. Indeed, note that  $\phi$ extends to a meromorphic quadratic differential on $\mathbb{CP}^1$ with a pole of order $n+4$ at $\infty$. The structure of the foliations at poles are understood (see $\S 6$ and $\S 7$, as well as figures 11, 12, and 13 in \cite{St}), and one can see from these structures that $\pi$ is injective in a neighbourhood of $\infty$. It follows that $\pi$ is proper. Putting the pieces together, we have that $\pi$ is a covering map onto its image. Since every vertex has exactly one preimage, $\pi$ must be a homeomorphism. By properness, $\pi(\mathbb{C})$ is closed.
 \end{proof}
 Thus, $\pi(\partial\mathbb{D})$ is an embedded circle. Our goal now for this subsection is to prove the result below.
 \begin{prop}\label{polymin}
    $\pi|_{\overline{\mathbb{D}}}$ solves the Plateau problem for $\pi(\partial \mathbb{D}).$
\end{prop}
The key step in the proof is to establish the following very useful fact.
 \begin{lem}\label{loclip}
     There exists a Lipschitz retraction $r:X\to \pi(\mathbb{C}).$
 \end{lem}
 Assuming the lemma, we prove the main proposition.
 \begin{proof}[Proof of Proposition \ref{polymin}]
Suppose that $g$ is another candidate for the Plateau problem. For topological reasons, $g(\overline{\mathbb{D}})$ contains $\pi(\mathbb{D})$. Indeed, suppose for contradiction's sake that $g(\overline{\mathbb{D}})$ misses a point in $ \pi(\mathbb{D}).$   Using the retraction map $r$, we retract $g(\overline{\mathbb{D}})$ to a subset of $\pi(\mathbb{C})$. Using that $\pi$ is a homeomorphism, we can then build a retraction from the image of $r\circ g$ to $\pi(\overline{\partial\mathbb{D}})$. Precomposing $g$ with a homeomorphism of $\overline{\mathbb{D}}$ in order to assume that $g|_{\partial\mathbb{D}}=\pi|_{\partial\mathbb{D}}$, and then post-composing $r\circ g$ by $\pi^{-1},$ we obtain a retraction from $\overline{\mathbb{D}}$ to $\partial\mathbb{D}$, which is of course impossible.

As a consequence of this inclusion, we have that $A(\pi|_{\overline{\mathbb{D}}})\leq A(g)$.
To conclude that $\pi|_{\overline{\mathbb{D}}}$ solves this Plateau problem, we simply use that energy dominates area: $$\mathcal{E}(\mathbb{D},\pi)= A(\pi|_{\overline{\mathbb{D}}})\leq A(g)\leq \mathcal{E}(\mathbb{D},g).$$
\end{proof}
 \begin{figure}[ht]
     \centering
     \includegraphics[scale=0.25]{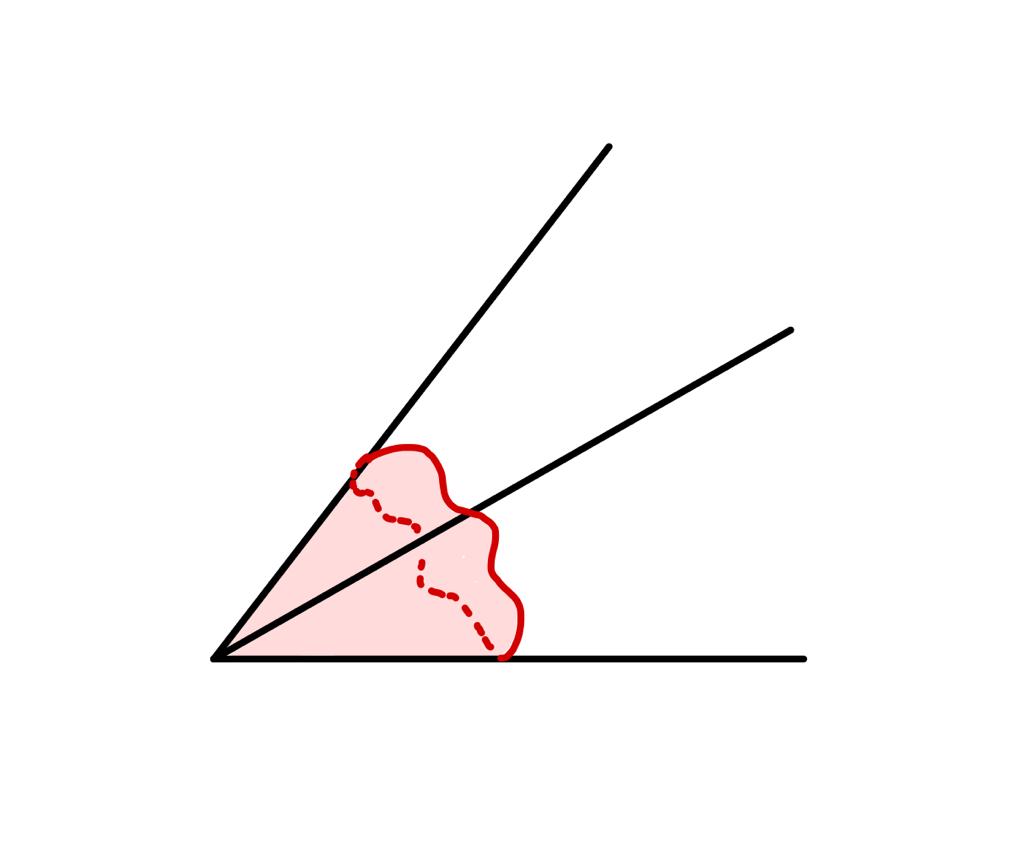}
    \caption{An embedded circle bounding a disk in a particularly simple product of simplicial complexes.}
     \label{fig:fig2}
 \end{figure} 
 
 The remainder of this subsection is devoted to justifying Lemma \ref{loclip}. The retraction $r$ is easy to build by hand in some particular cases (we invite the reader to try for $\phi=z^n dz^2)$. To give a proof that won't become too complicated in the general case, we go through a result of Guirardel \cite{G} about retractions onto subsets of $X$. First recall some terminology from \cite{G}. Given any $\R$-tree $(T,d)$, a direction based at a point $x\in T$ is a connected component of $T\backslash\{x\}$. In the product of trees $X,$ a product of directions based at points $x$ and $y$ is called a quadrant based at $(x,y)$. Given a family of quadrants $\mathcal{Q}$, the core is the complement $\mathcal{C}_{\mathcal{Q}}=X\backslash \cup_{Q\in \mathcal{Q}}Q$. Also, for $i=1,2$, let $p_i: X\to T_i$ be the projection map.
\begin{prop}[Lemma 4.18 in \cite{G}]\label{4.18}
    Let $\mathcal{Q}$ be a family of quadrants whose core $\mathcal{C}_{\mathcal{Q}}$ is connected and such that for $i=1,2,$ $p_i(\mathcal{C}_{\mathcal{Q}})=T_i.$ Then there exists a semi-flow $\varphi_t: X\to X$, $t\in [0,\infty)$, that restricts to the identity on $\mathcal{C}_{\mathcal{Q}}$, is $2(1+\sqrt{2})$-Lipschitz for all $t$, and such that for all $x\in X$ there exists a $t>0$ such that $\varphi_t(x)\in\mathcal{C}_{\mathcal{Q}}.$
 \end{prop}
\begin{remark}
     In Guirardel's setting, there is a group $G$ acting on each of the trees, and the semi-flow is $G$-equivariant. However, there is no assumption on $G$ nor on the action, and so we can take $G$ to be trivial. 
\end{remark}
The Lipschitz constant is not explicitly stated in \cite[Lemma 4.18]{G}, so we will explain here. Going through the proof, one finds that at each fixed $t$ and $x_2\in T_2$, $\varphi_t$ satisfies, $d_\infty(\varphi_t(x_1,x_2),\varphi_t(x_1',x_2))\leq (1+\sqrt{2})(2d_1)(x_1,x_2),$ where $d_{\infty}=\max_{i=1,2}2d_i$, and likewise for fixed $t$ and $x_1\in T_1.$ From the triangle inequality, we have that for points $(x_1,x_2),(y_1,y_2)\in X$, $$d_\infty(\varphi_t(x_1,x_2),\varphi_t(y_1,y_2)) \leq (1+\sqrt{2})(2d_1(x_1,y_1)+2d_2(x_2,y_2))\leq \sqrt{2}(1+\sqrt{2})(2d_1\times 2d_2)((x_1,x_2),(y_1,y_2)).$$ From the fact that $(2d_1\times 2d_2)\leq \sqrt{2}d_\infty,$ we conclude that each $\varphi_t$ is $2(1+\sqrt{2})$-Lipschitz.

\begin{remark}
    We'd like to address a possible ambiguity in the proof of \cite[Lemma 4.18]{G}. Guirardel uses his Proposition 4.6, which applies to families of quadrants in which no two quadrants face each other (see \cite[Definition 4.1]{G}). But it is easily verified that if $p_i(\mathcal{C}_{\mathcal{Q}})=T_i$ for $i=1,2,$ then no two quadrants face each other (if $x\in X$ is contained in the intersection of the facing quadrants, then $p_i(x)\not \in p_i(\mathcal{C}_{\mathcal{Q}})).$
\end{remark}

\begin{prop}\label{constructing r}
    Let $\mathcal{Q}$ be a family of quadrants whose core $\mathcal{C}_{\mathcal{Q}}$ is connected and such that for $i=1,2,$ $p_i(\mathcal{C}_{\mathcal{Q}})=T_i.$ Then there exists a Lipschitz retraction $r:X\to \mathcal{C}_{\mathcal{Q}}.$
\end{prop}
\begin{proof}
    Let $\varphi_t$ be the semi-flow of Proposition \ref{4.18}. Since $\varphi_t$ is uniformly Lipschitz for all $t$, as $t\to\infty$ it converges on compacta to a Lipschitz retraction $r:X\to \mathcal{C}_{\mathcal{Q}}$. Although it is not needed for the proof, we have convergence rather than sub-convergence because for every $x,$ there exists $T>0$ such that $\varphi_t(x)=\varphi_T(x)$ for all $t\geq T.$
\end{proof}

To prove Lemma \ref{loclip} we just need to verify that $\pi(\mathbb{C})$ satisfies the hypothesis of Proposition \ref{constructing r}. To carry out this verification, we use another result from \cite{G}. We say that a subset $E\subset X$ has convex fibers (or connected fibers) if for all $x\in T_i,$ $i=1,2$, the set $p_i^{-1}(x)\cap E$ is convex (this allows for empty fibers). We point out that in a tree, a connected set is convex.

\begin{lem}[Lemma 5.4 in \cite{G}]\label{Glem}
    Let $T_1',T_2'$ be two $\R$-trees and $F$ a non-empty connected subset of $T_1'\times T_2'$ with convex fibers. Then the complement of $\overline{F}$ is a family of quadrants.
\end{lem}
To apply Lemma \ref{Glem} to $\pi(\mathbb{C})$, we need to observe the following.
\begin{lem}\label{convexfibers}
    $\pi(\mathbb{C})\subset X$ has convex fibers.
\end{lem}
\begin{proof}
   We assume without loss of generality that $i=1$. Let $x\in T_1$ and let $\ell$ be the leaf in $\mathbb{C}$ projecting to the point $x$ under $\pi_1$. Then $p_1^{-1}(x)\cap E=\pi(\pi_1^{-1}(x))=\pi(\ell)=\{x\}\times \pi_2(\ell).$ Since $\ell$ is connected, $\{x\}\times \pi_2(\ell)$ is connected and moreover convex.
   \end{proof}

   \begin{proof}[Proof of Lemma \ref{loclip}]
       By Lemma \ref{convexfibers}, and since $\pi(\mathbb{C})$ is connected and closed, we see that there is a collection of quadrants $\mathcal{Q}$ such that $\pi(\mathbb{C})=\mathcal{C}_{\mathcal{Q}}$. As both projections $p_i$ satisfy $p_i(\pi(\mathbb{C}))=T_i$, we can apply Proposition \ref{constructing r} to obtain the retraction $r$.
   \end{proof}

\subsection{Proof of the new main inequality}
Combining Proposition \ref{newmaincrit} and Proposition \ref{polymin}, we obtain
\begin{prop}\label{newmainpoly}
    Let $\phi$ be a polynomial quadratic differential on $\mathbb{C}$. Then, for $\phi,-\phi$ and any choice of quasiconformal maps from $\mathbb{D}\to\mathbb{D}$ that agree on $\partial\mathbb{D},$ (\ref{newmain1}) holds.
\end{prop}
To promote Proposition \ref{newmainpoly} to Theorem B we use the following basic fact.
\begin{lem}\label{density result}
    Polynomials are dense in $A^1(\mathbb{D}).$
\end{lem}
\begin{proof}
    By Stone-Weierstrass, polynomials are dense in the Banach space of continuous functions on $\overline{\mathbb{D}}$ with the sup norm. A simple $\frac{\epsilon}{3}$-argument then shows density in $\mathbb{A}^1(\mathbb{D}).$
\end{proof}
It's now time to prove Proposition \ref{densityarg}. Recall the statement is that if (\ref{newmain1}) holds on an arbitrary dense subset $B\subset \mathbb{A}^1(\mathbb{D})$, then (\ref{newmain1}) holds on all of $\mathbb{A}^1(\mathbb{D})$.
\begin{proof}[Proof of Proposition \ref{densityarg}]
    Let $\phi\in \mathbb{A}^1(\mathbb{D})$, and let $(p_n)_{n=1}^\infty\subset B$ be a sequence approximating $\phi$ in $\mathbb{A}^1(\mathbb{D})$. Let $\mu_1$ and $\mu_2$ be Beltrami forms of quasiconformal maps on the unit disk with the same boundary values, and let $k=\max \{||\mu_1||_{L^{\infty}(\mathbb{D})}, ||\mu_2||_{L^{\infty}(\mathbb{D})}\}<1.$ By our hypothesis, for $\mu_1,\mu_2$ and every $p_n$: $$\textrm{Re}\int_{\mathbb{D}} \Big (p_n \frac{\mu_1}{1-|\mu_1|^2}-p_n \frac{\mu_2}{1-|\mu_2|^2}\Big ) \leq \int_{\mathbb{D}} |p_n|\Big ( \frac{|\mu_1|^2}{1-|\mu_1|^2}+\frac{|\mu_2|^2}{1-|\mu_2|^2}\Big ).$$
We check that for both $i=1,2$, $$\Big |\textrm{Re}\int_{\mathbb{D}} \phi \frac{\mu_i}{1-|\mu_i|^2} - \textrm{Re}\int_{\mathbb{D}}p_n \frac{\mu_i}{1-|\mu_i|^2} \Big | \leq \frac{k}{1-k^2}\int_{\mathbb{D}} |\phi-p_n| \to 0$$ as $n\to\infty$. Consequently,
\begin{equation}\label{rhs}
    \textrm{Re}\int_{\mathbb{D}} \Big (\phi\frac{\mu_1}{1-|\mu_1|^2}-\phi \frac{\mu_2}{1-|\mu_2|^2}\Big ) \leq \liminf_{n\to\infty} \int_{\mathbb{D}} |p_n|\Big ( \frac{|\mu_1|^2}{1-|\mu_1|^2}+\frac{|\mu_2|^2}{1-|\mu_2|^2}\Big ).
\end{equation}
For the right hand side, $$\int_{\mathbb{D}} |p_n|\Big ( \frac{|\mu_1|^2}{1-|\mu_1|^2}+\frac{|\mu_2|^2}{1-|\mu_2|^2}\Big ) \leq \int_{\mathbb{D}} |\phi| \Big ( \frac{|\mu_1|^2}{1-|\mu_1|^2}+\frac{|\mu_2|^2}{1-|\mu_2|^2}\Big ) + \frac{2k^2}{1-k^2}\int_{\mathbb{D}} |\phi-p_n|.$$ Taking $n\to\infty$ in (\ref{rhs}) thus yields $$\textrm{Re}\int_{\mathbb{D}} \Big (\phi\frac{\mu_1}{1-|\mu_1|^2}-\phi \frac{\mu_2}{1-|\mu_2|^2}\Big ) \leq \int_{\mathbb{D}} |\phi| \Big ( \frac{|\mu_1|^2}{1-|\mu_1|^2}+\frac{|\mu_2|^2}{1-|\mu_2|^2}\Big ),$$ which is exactly (\ref{newmain1}) for $n=2$.
\end{proof}
Putting everything together, we obtain Theorem B.
\begin{proof}[Proof of Theorem B]
By Proposition \ref{newmainpoly} and Lemma \ref{density result}, the new main inequality holds on a dense subset of $\mathbb{A}^1(\mathbb{D})$. We then apply Proposition \ref{densityarg}.
\end{proof}

\subsection{Proof of uniqueness}
We're now prepared to prove Theorem A. We do so following the argument of Markovi{\'c} for closed surfaces in \cite{M1}. We have to deal with a minor complication, which is that our harmonic maps might have infinite total energy. Indeed this is often the case: the harmonic diffeomorphisms in the hyperbolic case always have infinite energy (from \cite[Proposition 10]{Wan1}). Accordingly, we don't want to globally integrate any term that involves an energy density. 

Before we begin, we record a consequence of the definitions (\ref{en}) and (\ref{mhopf}): for any immersion $h$ to a Riemannian manifold of dimension at least $2$,
\begin{equation}\label{eineq}
    e(h)> 2|\phi(h)|
\end{equation}
everywhere, in the sense of tensors as before.

\begin{proof}[Proof of Theorem A]
Let $(h_1,h_2)$ and $(g_1,g_2)$ be two quasiconformal minimal maps to $(\mathbb{D}^2,\sigma_1\oplus\sigma_2)$ that induce minimal diffeomorphisms with $L^1$ Hopf differentials and that extend to the boundary as the same map $h_1\circ h_2^{-1}=g_1\circ g_2^{-1}=\varphi$. 
For $i=1,2$, set $f_i = g_i^{-1}\circ h_i.$ The extensions of $f_1$ and $f_2$ to $\partial\mathbb{D}$ agree, so the new main inequality can be used on the Beltrami forms $\mu_i$ of $f_i$. Suppose for the sake of contradiction that for at least one $i$, $f_i$ is not the identity. Let $\phi=\phi_1$ be the Hopf differential of $h_1$, so that $-\phi=\phi_2$ is the Hopf differential of $h_2$. To simplify notation, for each $0<r<1,$ set $\mathbb{D}_r=\{z\in\mathbb{C}: |z|<r\}.$ Applying the Reich-Strebel formula (\ref{RSorig}) over the subsurface $\mathbb{D}_r,$
\begin{align*}
    \sum_{i=1}^2 \mathcal{E}(f_i(\mathbb{D}_r),g_i)-\sum_{i=1}^2\mathcal{E}(\mathbb{D}_r,h_i) &= \sum_{i=1}^2 \mathcal{E}(f_i(\mathbb{D}_r),h_i\circ f_i^{-1})-\sum_{i=1}^2\mathcal{E}(\mathbb{D}_r,h_i) \\
    &= \sum_{i=1}^2 -4\textrm{Re}\int_{\mathbb{D}_r}\phi_i \frac{\mu_i}{1-|\mu_i|^2}+ \sum_{i=1}^2 2\int_{\mathbb{D}_r}e(h_i) \frac{|\mu_i|^2}{1-|\mu_i|^2} \\
    &> \sum_{i=1}^2 -4\textrm{Re}\int_{\mathbb{D}_r}\phi_i \frac{\mu_i}{1-|\mu_i|^2}+ \sum_{i=1}^2 4\int_{\mathbb{D}_r}|\phi_i| \frac{|\mu_i|^2}{1-|\mu_i|^2},
\end{align*}
where in the last line we used (\ref{eineq}). Note that difference between the last two lines is increasing with $r,$ so in particular once $r$ is at least $1/2$, there is an $\epsilon>0$ independent of $r$ such that $$\sum_{i=1}^2 \mathcal{E}(f_i(\mathbb{D}_r),g_i)-\sum_{i=1}^2\mathcal{E}(\mathbb{D}_r,h_i)\geq \sum_{i=1}^2 -4\textrm{Re}\int_{\mathbb{D}_r}\phi_i \frac{\mu_i}{1-|\mu_i|^2}+ \sum_{i=1}^2 4\int_{\mathbb{D}_r}|\phi_i| \frac{|\mu_i|^2}{1-|\mu_i|^2}+\epsilon.$$ Since each $\phi_i$ is integrable, $$\sum_{i=1}^2 4\textrm{Re}\int_{\mathbb{D}_r}\phi_i \frac{\mu_i}{1-|\mu_i|^2}\to \sum_{i=1}^2 4\textrm{Re}\int_{\mathbb{D}}\phi_i \frac{\mu_i}{1-|\mu_i|^2}$$ and $$\sum_{i=1}^2 4\int_{\mathbb{D}_r}|\phi_i| \frac{|\mu_i|^2}{1-|\mu_i|^2}\to \sum_{i=1}^2 4\int_{\mathbb{D}}|\phi_i| \frac{|\mu_i|^2}{1-|\mu_i|^2}$$ as $r$ increases to $1$. Using Theorem B, we deduce that $$\sum_{i=1}^2 \mathcal{E}(f_i(\mathbb{D}_r),g_i)-\sum_{i=1}^2\mathcal{E}(\mathbb{D}_r,h_i)\geq \epsilon$$ for $r$ sufficiently close to $1$.  Reversing the roles of $h_i$ and $g_i$, we repeat the argument (while choosing our domains carefully): if $\psi_i$ is the Hopf differential of $g_i$ and $\nu_i$ is the Beltrami form of $f_i^{-1}$, there is a $\delta>0$ such that for $r\geq 1/2,$
\begin{align*}
    \sum_{i=1}^2 \mathcal{E}(\mathbb{D}_r,h_i)-\sum_{i=1}^2\mathcal{E}(f_i(\mathbb{D}_r),g_i)&= \sum_{i=1}^2 \mathcal{E}(\mathbb{D}_r,g_i\circ f_i)-\sum_{i=1}^2\mathcal{E}(f_i(\mathbb{D}_r),g_i) \\
    &\geq \sum_{i=1}^2 -4\textrm{Re}\int_{f_i(\mathbb{D}_r)}\psi_i \frac{\nu_i}{1-|\nu_i|^2}+ \sum_{i=1}^2 4\int_{f_i(\mathbb{D}_r)}|\psi_i| \frac{|\nu_i|^2}{1-|\nu_i|^2} +\delta.
\end{align*}
Since $f_i$ is quasiconformal, it is not hard to justify that $$\sum_{i=1}^2 4\textrm{Re}\int_{f_i(\mathbb{D}_r)}\psi_i \frac{\nu_i}{1-|\nu_i|^2}\to \sum_{i=1}^2 4\textrm{Re}\int_{\mathbb{D}}\psi_i \frac{\nu_i}{1-|\nu_i|^2}$$ as $r\to 1,$ and likewise for the other term. We deduce that
$$\sum_{i=1}^2 \mathcal{E}(\mathbb{D}_r,h_i)-\sum_{i=1}^2\mathcal{E}_r(f_i(\mathbb{D}_r),g_i)\geq \delta$$ for $r$ close to $1$, which gives the contradiction. We conclude that both $f_i$ are the identity map, so that $(h_1,h_2)=(g_1,g_2).$
\end{proof}

\section{The general situation}\label{generalsit}

\subsection{Counterexample without integrability}\label{counter}
We now present Example A, showing that Theorem A is not true without the $L^1$ condition. The minimal diffeomorphisms are simple enough to write out by hand.

\begin{proof}[Example A]
   Let $m$ be the flat Euclidean metric on the upper halfspace $\mathbb{H}\subset \mathbb{C}$. We give two minimal maps $(\mathbb{H},m)\to (\mathbb{H},m)$ with the same boundary data, one of which has Hopf differential $\phi$ with $\int_{\mathbb{H}}|\phi| = \infty$. Conjugating via the Cayley transform $\mathcal{C}$ gives an example of non-uniqueness on the unit disk with the flat metric $|\mathcal{C}'(z)||dz|^2.$ 

Let $\zeta=\xi+i\eta$ be the coordinate on $\mathbb{H}$. For $k>0$ and $k\neq 1$, set $$h_k(\zeta) = \xi + ik\eta, \hspace{1mm} g_k(\zeta)=k\xi+i\eta.$$ Then, $\phi(h_k) = (1-k^2)d\zeta^2$ and $\phi(g_k) = (k^2-1)d\zeta^2,$ so that $f_k=h_k\circ g_k^{-1}$ is a minimal diffeomorphism, explicitly given by $$f_k=h_k\circ g_k^{-1}(\zeta)=k^{-1}\xi + ik\eta.$$ $f_k|_{\partial\mathbb{H}}$ is the map $\xi\mapsto k^{-1}\xi.$ Next, consider the conformal maps $$\varphi_k(\zeta) = k^{-1/2}(\xi + i\eta), \hspace{1mm} \psi_k(\zeta) = k^{1/2}(\xi + i\eta).$$ The map $\tau_k = \varphi_k\circ \psi_k^{-1}$ is simply multiplication by $k^{-1},$ both conformal and a minimal diffeomorphism, and it agrees with $f_k$ on $\partial\mathbb{H}$. 
\end{proof}
For the Hopf differentials $\phi(h_k)$ and $\phi(g_k)$ of $f_k$, the leaf space projections identify the product of $\R$-trees with the upper halfspace $\mathbb{H}$ equipped with the flat metric $4|k^2-1|m$. In contrast to the $\R$-trees from Section \ref{plateausection}, this space admits lots of minimal maps that extend to the boundary. 

Analytically, we take note of the following aspects. Let $\rho=\rho(z)|dz|^2$ be the hyperbolic metric on $\mathbb{D}$ of constant curvature $-1$. The Bers norm of a holomorphic quadratic differential $\phi=\phi(z)dz^2$ on $\mathbb{D}$ is $$||\phi||_\rho^2=\sup_{z\in\mathbb{D}}|\phi(z)|^2\rho^{-2}(z).$$ As we mentioned in Remark \ref{Bers}, if $\sigma$ is a complete metric on $\mathbb{D}$ with pinched negative curvature and $f:\mathbb{D}\to (\mathbb{D},\sigma)$ is a quasiconformal harmonic diffeomorphism, then the Hopf differential of $f$ has finite Bers norm. Turning to the example above, the Bers norm of $\phi(h_k)$ blows up as $z\to -i$, since the Cayley transform pulls back $y^{-2}|d\zeta|^2$ to $\rho(z)|dz|^2$, and $$|\phi(h_k)|y^2\to \infty$$ as $y\to\infty$. Thus, our example reflects the failure of Remark \ref{Bers} in general.

It is proved in \cite{AMM} that a quadratic differential $\phi$ on $\mathbb{D}$ has finite Bers norm if and only if its maximal $\phi$-disks (see \cite{AMM} for the definition) have bounded radius, which implies that if we carry out the leaf-space constructions for $\phi$, then the image of any ball not containing a zero of $\phi$ under the product of the leaf-space projections must have bounded radius. This suggests to us that, generally speaking, minimal maps to products of $\R$-trees arising from such differentials should have more controlled geometry. These observations provide evidence for a positive resolution to Problem A.

\subsection{Stability}
In the general case we can prove a stability result, Theorem A'. Given a surface $\Sigma$ with Riemannian metrics $\sigma_1,\sigma_2$, a diffeomorphism $f$ of $\Sigma,$ and a relatively compact open subset $U$ of $\Sigma$, let $A(f|_U)$ be the area of the graph of $f$ in the product $(\Sigma^2,\sigma_1\oplus\sigma_2).$
\begin{defn}\label{stabilitydef}
    A minimal diffeomorphism $f:(\Sigma,\sigma_1)\to (\Sigma,\sigma_2)$ is stable if for every relatively compact open subset $U$ of $\Sigma$ and $C^\infty$ path of diffeomorphisms $f_t: (\Sigma,\sigma_1)\to (\Sigma,\sigma_2)$ such that $f_0=f$ and $f_t=f$ on the complement of $U$, 
$$\frac{d^2}{dt^2}|_{t=0}A(f_t|_U) \geq 0.$$
\end{defn} 
\begin{remark}
    Since we're starting at a critical point for the area, our definition of stability is equivalent to asking that the second derivative of area at time zero be non-negative along any path $f_t$ that's tangent to a variation with compact support in $U.$
\end{remark}
Recall that Theorem A' says that any minimal diffeomorphism is stable. We prove the theorem by taking the second derivative on the new main inequality.
\begin{proof}[Proof of Theorem A']
    Let $f_t:(\mathbb{D},\sigma_1)\to (\mathbb{D},\sigma_2)$ be a path of diffeomorphisms with $f_0=f$ and $f_t=f$ on the complement of a relatively compact subset $U$. It does no harm to enlarge $U$, so let's assume it is a disk centered at the origin. 
    Let $(h_1,h_2)$ be the initial minimal map to the product of disks, and $t\mapsto h_t=(h_1^t,h_2^t)$ the path of minimal maps engendered by the path of diffeomorphisms $t\mapsto f_t.$ Note that stability is equivalent to energy stability: it suffices to show that $$\frac{d^2}{dt^2}|_{t=0} \sum_{i=1}^2\mathcal{E}(U, h_i^t)\geq 0.$$ Indeed, if area is lowered to second order, then by a well-known application of the measurable Riemann mapping theorem, one can find a path of quasiconformal maps $g_t:U\to U$ starting at the identity and such that $h_t\circ g_t$ is conformal. For every $t,$ $\mathcal{E}(U,h_t\circ g_t)=A(h_t\circ g_t|_U),$ and hence energy is lowered to second order along the modified variation $h_t\circ g_t$. The other side of the equivalence (which is not needed for our proof) is clear because energy dominates area, and minimal maps are critical points for both energy and area. 
    
    We now compute in the same manner as the proof of Theorem A. Let $\phi_1=\phi$ be the Hopf differential of $h_1$ and $\phi_2=-\phi$ the differential of $h_2$. For each $i=1,2$ and $t>0,$ set $v_i^t=(h_i^t)\circ h_i^{-1}$ and let $\mu_i^t$ be the Beltrami form of $v_i^t,$ which vanishes outside of $U$. For all $t$ we have that $h_1^t=h_2^t$ on $\partial U$, hence $v_1^t=v_2^t$ on $\partial U$, and moreover the new main inequality applies to $\phi$, $-\phi,$ and $\mu_1^t,\mu_2^t$ in restriction to $\overline{U}$. Using $h_i^t=h_i\circ (v_i^t)^{-1}$, (\ref{RSorig}), and (\ref{eineq}), 
\begin{equation}\label{final}
    \sum_{i=1}^2\mathcal{E}(U, h_i^t) - \sum_{i=1}^2\mathcal{E}(U,h_i) \geq\sum_{i=1}^2 -4\textrm{Re}\int_{U}\phi_i \frac{\mu_i^t}{1-|\mu_i^t|^2}+ \sum_{i=1}^2 4\int_{U}|\phi_i| \frac{|\mu_i^t|^2}{1-|\mu_i^t|^2}
\end{equation}
 with equality if and only if both $\mu_i^t$ are zero almost everywhere. For each fixed $t$, Theorem B returns that the right hand side of (\ref{final}) is non-negative. Since the initial map $(h_1,h_2)$ is minimal, or really just because $\phi_1+\phi_2=0,$ the first derivatives of both sides of (\ref{final}) are zero. It follows that the left hand side of (\ref{final}) has non-negative second derivative: 
    \begin{equation}\label{finalderi}
        \frac{d^2}{dt^2}|_{t=0} \Big (\sum_{i=1}^2\mathcal{E}(U, h_i^t) - \sum_{i=1}^2\mathcal{E}(U,h_i)\Big )\geq 0.
    \end{equation}
    We conclude by noting that $$\frac{d^2}{dt^2}|_{t=0}  \sum_{i=1}^2\mathcal{E}(U, h_i^t)=\frac{d^2}{dt^2}|_{t=0} \Big (\sum_{i=1}^2\mathcal{E}(U, h_i^t) - \sum_{i=1}^2\mathcal{E}(U,h_i)\Big )$$ and applying (\ref{finalderi}).

\end{proof}

\bibliographystyle{plain}
\bibliography{bibliography}

\begin{thebibliography}{10}

\bibitem{AMM}
I.~Ani\'{c}, V.~Markovi\'{c}, and M.~Mateljevi\'{c}.
\newblock Uniformly bounded maximal $\varphi$-disks, bers space and harmonic maps.
\newblock {\em Proc. Amer. Math. Soc.}, 128:2947--2956, 2000.

\bibitem{ATW}
Thomas K.~K. Au, Luen-Fai Tam, and Tom Y.~H. Wan.
\newblock Hopf differentials and the images of harmonic maps.
\newblock {\em Comm. Anal. Geom.}, 10(3):515--573, 2002.

\bibitem{AW}
Thomas Kwok-Keung Au and Tom Yau-Heng Wan.
\newblock Prescribed horizontal and vertical trees problem of quadratic differentials.
\newblock {\em Commun. Contemp. Math.}, 8(3):381--399, 2006.

\bibitem{BH}
Yves Benoist and Dominique Hulin.
\newblock {Cubic differentials and hyperbolic convex sets}.
\newblock {\em J. Differential Geom.}, 98(1):1 -- 19, 2014.

\bibitem{BS}
Francesco Bonsante and Jean-Marc Schlenker.
\newblock Maximal surfaces and the universal {T}eichm\"{u}ller space.
\newblock {\em Invent. Math.}, 182(2):279--333, 2010.

\bibitem{BSe}
Francesco Bonsante and Andrea Seppi.
\newblock Area-preserving diffeomorphisms of the hyperbolic plane and {$K$}-surfaces in anti-de {S}itter space.
\newblock {\em J. Topol.}, 11(2):420--468, 2018.

\bibitem{Bre}
Simon Brendle.
\newblock Minimal {L}agrangian diffeomorphisms between domains in the hyperbolic plane.
\newblock {\em J. Differential Geom.}, 80(1):1--22, 2008.

\bibitem{CM}
Tobias~Holck Colding and William~P. Minicozzi, II.
\newblock {\em A course in minimal surfaces}, volume 121 of {\em Graduate Studies in Mathematics}.
\newblock American Mathematical Society, Providence, RI, 2011.

\bibitem{DL}
Song Dai and Qiongling Li.
\newblock Domination results in {$n$}-{F}uchsian fibers in the moduli space of {H}iggs bundles.
\newblock {\em Proc. Lond. Math. Soc. (3)}, 124(4):427--477, 2022.

\bibitem{FW}
Benson Farb and Michael Wolf.
\newblock Harmonic splittings of surfaces.
\newblock {\em Topology}, 40(6):1395--1414, 2001.

\bibitem{Thbook}
Albert Fathi, Fran\c{c}ois Laudenbach, and Valentin Po\'{e}naru.
\newblock {\em Thurston's work on surfaces}, volume~48 of {\em Mathematical Notes}.
\newblock Princeton University Press, Princeton, NJ, 2012.
\newblock Translated from the 1979 French original by Djun M. Kim and Dan Margalit.

\bibitem{G}
Vincent Guirardel.
\newblock C\oe ur et nombre d'intersection pour les actions de groupes sur les arbres.
\newblock {\em Ann. Sci. \'{E}c Norm. Sup\'{e}r. (4)}, 38(6):847--888, 2005.

\bibitem{Hu}
John~Hamal Hubbard.
\newblock {\em Teichm\"{u}ller theory and applications to geometry, topology, and dynamics. {V}ol. 1}.
\newblock Matrix Editions, Ithaca, NY, 2006.
\newblock Teichm\"{u}ller theory, With contributions by Adrien Douady, William Dunbar, Roland Roeder, Sylvain Bonnot, David Brown, Allen Hatcher, Chris Hruska and Sudeb Mitra, With forewords by William Thurston and Clifford Earle.

\bibitem{KS}
Nicholas~J. Korevaar and Richard~M. Schoen.
\newblock Sobolev spaces and harmonic maps for metric space targets.
\newblock {\em Comm. Anal. Geom.}, 1(3-4):561--659, 1993.

\bibitem{LT}
François Labourie and Jérémy Toulisse.
\newblock Quasicircles and quasiperiodic surfaces in pseudo-hyperbolic spaces, 2022. To appear in Invent. Math.

\bibitem{LTW}
François Labourie, Jérémy Toulisse, and Michael Wolf.
\newblock Plateau problems for maximal surfaces in pseudo-hyperbolic spaces, 2022. To appear in Ann. Sci. Éc. Norm. Supér.

\bibitem{LWenger}
Alexander {Lytchak} and Stefan {Wenger}.
\newblock {Area Minimizing Discs in Metric Spaces}.
\newblock {\em Archive for Rational Mechanics and Analysis}, 223(3):1123--1182, March 2017.

\bibitem{MM}
V.~Markovi\'{c} and M.~Mateljevi\'{c}.
\newblock A new version of the main inequality and the uniqueness of harmonic maps.
\newblock {\em J. Anal. Math.}, 79:315--334, 1999.

\bibitem{M1}
Vladimir Markovi\'{c}.
\newblock Uniqueness of minimal diffeomorphisms between surfaces.
\newblock {\em Bull. Lond. Math. Soc.}, 53(4):1196--1204, 2021.

\bibitem{M2}
Vladimir Markovi\'{c}.
\newblock Non-uniqueness of minimal surfaces in a product of closed {R}iemann surfaces.
\newblock {\em Geom. Funct. Anal.}, 32(1):31--52, 2022.

\bibitem{MS}
Vladimir Markovic and Nathaniel Sagman.
\newblock Minimal surfaces and the new main inequality.
\newblock {\em arXiv preprint arXiv:2301.00249}, 2022.

\bibitem{MSS}
Vladimir Markovic, Nathaniel Sagman, and Peter Smillie.
\newblock Unstable minimal surfaces in $\mathbb{R}^n$ and in products of hyperbolic surfaces.
\newblock {\em arXiv preprint arXiv:2206.02938}, 2022.

\bibitem{Nit}
Johannes C.~C. Nitsche.
\newblock Plateau's problems and their modern ramifications.
\newblock {\em Amer. Math. Monthly}, 81:945--968, 1974.

\bibitem{RS}
Edgar Reich and Kurt Strebel.
\newblock On the {G}erstenhaber-{R}auch principle.
\newblock {\em Israel J. Math.}, 57(1):89--100, 1987.

\bibitem{SS}
Nathaniel Sagman and Peter Smillie.
\newblock Unstable minimal surfaces in symmetric spaces of non-compact type.
\newblock {\em arXiv preprint arXiv:2208.04885}, 2022.

\bibitem{Sc}
Richard~M. Schoen.
\newblock The role of harmonic mappings in rigidity and deformation problems.
\newblock In {\em Complex geometry ({O}saka, 1990)}, volume 143 of {\em Lecture Notes in Pure and Appl. Math.}, pages 179--200. Dekker, New York, 1993.

\bibitem{SST}
Andrea Seppi, Graham Smith, and Jérémy Toulisse.
\newblock On complete maximal submanifolds in pseudo-hyperbolic space, 2023.

\bibitem{St}
Kurt Strebel.
\newblock {\em Quadratic differentials}, volume~5 of {\em Ergebnisse der Mathematik und ihrer Grenzgebiete (3) [Results in Mathematics and Related Areas (3)]}.
\newblock Springer-Verlag, Berlin, 1984.

\bibitem{TW}
Luen-Fai Tam and Tom Y.-H. Wan.
\newblock Harmonic diffeomorphisms into {C}artan-{H}adamard surfaces with prescribed {H}opf differentials.
\newblock {\em Comm. Anal. Geom.}, 2(4):593--625, 1994.

\bibitem{Ta}
Andrea Tamburelli.
\newblock Polynomial quadratic differentials on the complex plane and light-like polygons in the {E}instein universe.
\newblock {\em Adv. Math.}, 352:483--515, 2019.

\bibitem{Tre}
Enrico Trebeschi.
\newblock Constant mean curvature hypersurfaces in anti-de sitter space, 2023.

\bibitem{Wa}
Tom Y.~H. Wan.
\newblock Stability of minimal graphs in products of surfaces.
\newblock In {\em Geometry from the {P}acific {R}im ({S}ingapore, 1994)}, pages 395--401. de Gruyter, Berlin, 1997.

\bibitem{Wan1}
Tom Yau-Heng Wan.
\newblock Constant mean curvature surface, harmonic maps, and universal {T}eichm\"{u}ller space.
\newblock {\em J. Differential Geom.}, 35(3):643--657, 1992.

\bibitem{We}
Richard~A. Wentworth.
\newblock Energy of harmonic maps and {G}ardiner's formula.
\newblock In {\em In the tradition of {A}hlfors-{B}ers. {IV}}, volume 432 of {\em Contemp. Math.}, pages 221--229. Amer. Math. Soc., Providence, RI, 2007.

\bibitem{Wf}
Michael Wolf.
\newblock On realizing measured foliations via quadratic differentials of harmonic maps to {$\mathbf{R}$}-trees.
\newblock {\em J. Anal. Math.}, 68:107--120, 1996.

\end{thebibliography}

\end{document}